\documentclass{amsart}
\usepackage{amsfonts}
\usepackage{cite}

\newtheorem{theorem}{Theorem}
\theoremstyle{plain}

\newtheorem{corollary}{Corollary}

\newtheorem{proposition}{Proposition}

\numberwithin{equation}{section}

\begin{document}
\title[Symmetries of generalized Robertson-Walker space-times]{On symmetries
of generalized Robertson-Walker space-times and applications}
\author{H. K. El-Sayied}
\address{Mathematics Department, Faculty of Science, Tanata University,
Tanta, Egypt}
\email{hkelsayied1989@yahoo.com}
\author{S. Shenawy}
\address{Modern Academy for engineering and Technology, Maadi, Egypt}
\email[S. Shenawy]{drssshenawy@eng.modern-academy.edu.eg, drshenawy@mail.com}
\urladdr{http://www.modern-academy.edu.eg}
\author{N. Syied}
\curraddr{Modern Academy for engineering and Technology, Maadi, Egypt}
\email[N. Sayed]{drnsyied@mail.com}
\urladdr{http://www.modern-academy.edu.eg}

\begin{abstract}
The purpose of the present article is to study and characterize several
types of symmetries of generalized Robertson-Walker space-times. Conformal
vector fields, curvature and Ricci collineations are studied. Many
implications for existence of these symmetries on generalied
Robertson-Walker spacetimes are obtained. Finally, Ricci solitons on
generalized Robertson-Walker space-times admitting conformal vector fields
are investigated.
\end{abstract}

\maketitle

\section{An introduction}

Robertson-Walker spacetimes have been extensively studied in both
mathematics and physics for a long time\cite{Besse2008, Chen2008,
Ivancevic2007, Sanchez2000, Sanchez1998, Sanchez1999}. This family of
spacetimes is a very important family of cosmological models in general
relativity\cite{Chen2008}. A generalized $\left( n+1\right) -$dimensional
Robertson-Walker (GRW) spacetime is a warped product manifold $I\times _{f}M$
where $M$ is an $n-$dimensional Riemannian manifold without any additional
assumptions on its fiber. The family of generalized Robertson-Walker
spacetimes widely extends the classical Robertson-Walker spacetimes $I\times
_{f}S_{k}$ where $S_{k}$ is a $3-$dimensional Riemannian manifold with
constant curvature.

The study of spacetime symmetries is of great interest in both mathematics
and physics. The existence of some symmetries in a spacetime is helpful in
solving Einstein field equation and in providing further insight to
conservative laws of dynamical systems(see \cite{Hall2004} one of the best
references for $4-$dimensional spacetime symmetries). Conformal vector
fields have been played an important role in both mathematics and physics%
\cite{Deshmokh2012, Deshmokh20141, Deshmokh20142, Kuhnel1997, Nomizo1960,
Yorozu1982}. The existence of a nontrivial conformal vector field is a
symmetry assumption for the metric tensor. This assumption has been widely
used in relativity to obtain exact solutions of the Einstein field equation%
\cite{Caballero2011}. Similarly, collineations display some tensors symmetry
properties of spacetimes. They are vector fields which preserve certain
feature of a spacetime(physical or geometric quantities such as matter and
curvature tensors) along their local flow lines. In this sense, the Lie
derivative of such quantities vanishes in direction of collineation vector
fields. Matter, curvature and Ricci collineations have been extensively
studied on spacetimes because of their essential role in general relativity.
Moreover, these collineations help to describe the geometry of spacetimes.
In the last two decates, an extensive work has been done studying
collineations and their generalizations such as Ricci inheritance
collineations on classical spacetimes. Among those, there are many authors
who studied these symmetries on classical Robertson-Walker spacetimes(for
instance see\cite{Duggala2005, Sanchez1999, Steller2006, Unal2012} and
references therein).

However, as far as we know, there is no study on generalized
Robertson-Walker spacetimes investigating neither conformal vector fields
nor different types of collineations up to this paper in which we intend to
fill this gab by providing many answers of the following questions: Under
what conditions is a vector field on $I\times _{f}M$ a conformal vector
field or a certain collineation? What does the fibre $M$ inherit from a
generalized Robertson-Walker spacetime $I\times _{f}M$ admitting a
collineation or a conformal vector field? The main purpose of the current
article is to study and explore both conformal vector fields and
collineations on generalized Robertson-Walker spacetimes. We gave a special
attention to two disjoint classes of conformal vector fields, namely,
Killing vector fields of constant length and concircular vector fields.
Finally, Ricci solitons on generalized Robertson-Walker spacetime admitting
either Killing or concircular vector fields are considered.

This article is organized as follows. The next section presents some
connection and curvature related formulas of generalized Robertson-Walker
spacetimes that are needed. Then basic definitions of conformal vector
fields and collineations are considered. Most of these results are
well-known and so proofs are omitted. Section 3 presents a study of
conformal vector fields and collineations on generalized Robertson-Walker
spacetimes. Finally, in section $4$, we study Ricci solitons on generalized
Robertson-Walker spacetimes admitting either Killing or concircular vector
fields.

\section{Preliminaries}

First, we want to fix some definitions and concepts. The warped product $%
M_{1}\times _{f}M_{2}$ of two Riemannian manifolds $\left(
M_{1},g_{1}\right) $ and $\left( M_{2},g_{2}\right) $ is the product
manifold $M_{1}\times M_{2}$ equipped with the metric tensor%
\begin{equation}
g=\pi _{1}^{\ast }\left( g_{1}\right) \oplus \left( f\circ \pi _{1}\right)
^{2}\pi _{2}^{\ast }\left( g_{2}\right)
\end{equation}%
with a smooth function $f:M_{1}\rightarrow \left( 0,\infty \right) $ where $%
\pi _{i}:M_{1}\times _{f}M_{2}\rightarrow M_{i}$ is the natural projection
map of the Cartesian product $M_{1}\times M_{2}$ onto $M_{i},$ $i=1,2$ and $%
\ast $ denotes the pull-back operator on tensors. The factors $\left(
M_{1},g_{1}\right) $ and $\left( M_{2},g_{2}\right) $ are usually called the
base manifold and fiber manifold respectively while $f$ is called as the
warping function\cite{Bishop1969, Oneill1983}. In particular, if $f=1$, then 
$M_{1}\times _{1}M_{2}=M_{1}\times M_{2}$ is the usual Cartesian product
manifold. It is clear that the submanifold $M_{1}\times \{q\}$ is isometric
to $M_{1}$ for every $q\in M_{2}$. Moreover, $\{p\}\times M_{2}$ is
homothetic to $M_{2}$ for every $p\in M_{1}$. Throughout this article we use
the same notation for a vector field and for its lift to the product
manifold.

Generalized Robertson-Walker spacetimes are well-known examples of warped
product spaces. A generalized Robertson-Walker spacetime is the warped
product $\bar{M}=I\times _{f}M$ with fiber $\left( M,g\right) $ any $n-$%
dimensional Riemannian manifold and base an open connected subinterval $%
\left( I,-\mathrm{d}t^{2}\right) $ of the real line $%
\mathbb{R}
$ endowed with the metric%
\begin{equation}
\bar{g}=-\mathrm{dt}^{2}\oplus f^{2}g
\end{equation}%
where $\mathrm{d}t^{2}$ is the Euclidean metric on $I$. The family of
generalized Robertson-Walker spacetimes $\bar{M}=I\times _{f}M$ widely
extends the classical Robertson-Walker spacetimes $I\times _{f}S_{k}$ where $%
M=S_{k}$ is a $3-$dimensional Riemannian manifold of constant sectional
curvature $k$. The warping function $f\left( t\right) $ is sometimes called
the scale factor. This factor tells us how big is the space-like slice at
sometime $t$.

For example the $\left( n+1\right) -$dimensional spherically symmetric
Friedmann--Robertson--Walker metric is given by%
\begin{equation*}
ds^{2}=-dt^{2}+f^{2}\left( t\right) \left( \frac{dr^{2}}{1-kr^{2}}%
+r^{2}d\Omega _{n-1}^{2}\right)
\end{equation*}%
where the spherical sector is given by $d\Omega _{n-1}^{2}=d\theta
_{1}^{2}+\sin ^{2}\theta _{1}d\theta _{2}^{2}+...+\sin ^{2}\theta
_{n-2}d\theta _{n-1}$ \cite{Garcia:2007}. The Einstein field equation for $%
\left( n+1\right) -$dimensional spacetime is given by%
\begin{equation*}
\mathrm{Ric}-\frac{r}{2}g=k_{n}T
\end{equation*}%
where $k_{n}$ is the multidimensional gravitational constant\cite%
{Garcia:2007}.

The following results are special cases of similar results on warped product
manifolds\cite{Bishop1969, Oneill1983,Shenawy:2016,Shenawy:2015}. Let $\bar{M%
}=I\times _{f}M$ be a generalized Robertson-Walker spacetime equipped with
the metric tensor $\bar{g}=-\mathrm{d}t^{2}\oplus f^{2}g$. Then the
Levi-Civita connection $\bar{D}$ on $\bar{M}$ is

\begin{equation}
\begin{tabular}{lll}
$\bar{D}_{\partial _{t}}\partial _{t}=0$ & \ \ \ \ \ \ \ \ \ \ \ \ \ \ \ \ \
\  & $\bar{D}_{\partial _{t}}X=\bar{D}_{X}\partial _{t}=\frac{\dot{f}}{f}X$
\\ 
\multicolumn{3}{l}{$\bar{D}_{X}Y=D_{X}Y-f\dot{f}g\left( X,Y\right) \partial
_{t}$}%
\end{tabular}
\label{Connection}
\end{equation}%
for any vector fields $X,Y\in \mathfrak{X}(M)$ where $D$ is the Levi-Civita
connection on $M$ and dots indicate differentiation with respect to $t$. The
curvature tensor \textrm{$\bar{R}$} of $\bar{M}$ is given by%
\begin{equation}
\begin{tabular}{lll}
\multicolumn{3}{l}{$\mathrm{\bar{R}}\left( \partial _{t},\partial
_{t}\right) \partial _{t}=\mathrm{\bar{R}}\left( \partial _{t},\partial
_{t}\right) X=\mathrm{\bar{R}}\left( X,Y\right) \partial _{t}=0$} \\ 
$\mathrm{\bar{R}}\left( X,\partial _{t}\right) \partial _{t}=-\frac{\ddot{f}%
}{f}X$ &  & $\mathrm{\bar{R}}\left( \partial _{t},X\right) Y=f\ddot{f}%
g\left( X,Y\right) \partial _{t}$ \\ 
\multicolumn{3}{l}{$\mathrm{\bar{R}}\left( X,Y\right) Z=\mathrm{R}\left(
X,Y\right) Z+\dot{f}^{2}\left[ g\left( X,Z\right) Y-g\left( Y,Z\right) X%
\right] $}%
\end{tabular}
\label{Curvature}
\end{equation}

where $\mathrm{R}$ is curvature tensor of $M$. Finally, the Ricci curvature
tensor $\mathrm{\bar{R}ic}$ on $\bar{M}$ is%
\begin{equation}
\begin{tabular}{lll}
$\mathrm{\bar{R}ic}\left( \partial _{t},\partial _{t}\right) =\frac{n\ddot{f}%
}{f}$ & \ \ \ \ \ \ \ \ \ \ \ \ \  & $\mathrm{\bar{R}ic}\left( X,\partial
_{t}\right) =0$ \\ 
\multicolumn{3}{l}{$\mathrm{\bar{R}ic}\left( X,Y\right) =$ \textrm{$Ric$}$%
\left( X,Y\right) -f^{\diamond }g\left( X,Y\right) $}%
\end{tabular}
\label{Ricci}
\end{equation}%
where $f^{\diamond }=-f\ddot{f}-\left( n-1\right) \dot{f}^{2}$.

Now, we will recall the definitions of conformal vector fields and some
collineations on an arbitrary pseudo-Riemannian manifold. Let $\left(
M,g,D\right) $ be a pseudo-Riemannian manifold with metric $g$ where $D$ is
the Levi-Civita connection on $M$. A vector field $\zeta \in \mathfrak{X}%
\left( M\right) $ is called a Killing vector field if%
\begin{equation*}
\mathcal{L}_{\zeta }g=0
\end{equation*}%
It is easy to show that

\begin{equation}
\left( \mathcal{L}_{\zeta }g\right) (X,Y)=g(D_{X}\zeta ,Y)+g(X,D_{Y}\zeta )
\end{equation}%
for any $X,Y\in \mathfrak{X}\left( M\right) $. By using symmetry of this
equation, we get that $\zeta $ is a Killing vector field if and only if%
\begin{equation*}
g(D_{X}\zeta ,X)=0
\end{equation*}%
for any vector field $X\in \mathfrak{X}\left( M\right) $. A
pseudo-Riemannian $n-$dimensional manifold has at most $n\left( n+1\right)
/2 $ independent Killing vector fields. The symmetry generated by Killing
vector fields is called isometry. A pseudo-Riemannian manifold which admits
a maximum such symmetry has a constant sectional curvature.

A vector field $\zeta $ is called a conformal vector field if%
\begin{equation*}
\mathcal{L}_{\zeta }g=\rho g
\end{equation*}%
for some smooth function $\rho :M\rightarrow 
\mathbb{R}
$. $\zeta $ is called homothetic if $\rho $ is constant and Killing if $\rho
=0$. Also, $\zeta $ is called a concircular vector field if%
\begin{equation*}
D_{X}\zeta =\rho X
\end{equation*}%
for any vector field $X\in \mathfrak{X}\left( M\right) $\cite{Chen2014}. Let 
$\zeta \in \mathfrak{X}\left( M\right) $ be a concircular vector field on $M$%
, then%
\begin{equation*}
\mathcal{L}_{\zeta }g(X,Y)=2\rho g(X,Y)
\end{equation*}%
i.e. $\zeta $ is a conformal vector field with conformal factor $2\rho $. A
concircular vector field is a parallel vector field if $\rho =0$. Moreover,
for a constant factor $\rho $, we have%
\begin{equation*}
R\left( X,Y\right) \zeta =0
\end{equation*}

A Riemannian manifold $M$ is said to admit a curvature collineation if the
Lie derivative of the curvature tensor $\mathrm{R}$ vanishes in the
direction of a vector field $\zeta \in \mathfrak{X}\left( M\right) $, that is%
\begin{equation*}
\mathcal{L}_{\zeta }\mathrm{R}=0
\end{equation*}%
where $\mathrm{R}$ is the Riemann curvature tensor. Likewise, $M$ is said to
admit a Ricci curvature collineation if there is a vector field $\zeta \in 
\mathfrak{X}\left( M\right) $ such that%
\begin{equation*}
\mathcal{L}_{\zeta }\mathrm{Ric}=0
\end{equation*}%
where $\mathrm{Ric}$ is the Ricci curvature tensor. It is clear that every
Killing vector field is a curvature collineation and every curvature
collineation is a Ricci curvature collineation. The converse is not
generally true. Finally, a spacetime $M$ is said to admit a matter
collineation if there is a vector field $\zeta \in \mathfrak{X}\left(
M\right) $ such that%
\begin{equation*}
\mathcal{L}_{\zeta }\mathrm{T}=0
\end{equation*}%
where $\mathrm{T}$ is the energy-momentum tensor. The Einstein's field
equation (with cosmological constant) is given by%
\begin{equation*}
\mathrm{Ric}-\frac{r}{2}g=\kappa \mathrm{T}-\lambda g
\end{equation*}%
where $r$ is the scalar curvature and $\lambda $ is the cosmological
constant. Suppose that $\zeta $ is a Killing vector field, then%
\begin{equation*}
\mathcal{L}_{\zeta }T=0
\end{equation*}%
i.e.$\zeta $ is a matter collineation field. Note that a matter collineation
need not be a Killing vector field.

\section{Symmetries of generalized Robertson-Walker spacetimes}

In this section, we investigate several types of symmetries of generalized
Robertson-Walker spacetimes. Necessary and sufficient conditions are derived
for a generalized Robertson-Walker spacetime to admit a conformal vector
field or a collineation.

We begin this section with the following well-known proposition\cite%
{Shenawy:2015}. Let $\bar{M}=I\times _{f}M$ \ be a generalized
Robertson-Walker spacetime equipped with the metric tensor $\bar{g}%
=-dt^{2}\oplus f^{2}g$.

\begin{proposition}
Suppose that $h\partial _{t},x\partial _{t},y\partial _{t}\in \mathfrak{X}%
(I) $ and $\zeta ,X,Y\in \mathfrak{X}(M)$, then%
\begin{equation}
\left( \mathcal{\bar{L}}_{\bar{\zeta}}\bar{g}\right) \left( \bar{X},\bar{Y}%
\right) =-2\dot{h}xy+f^{2}\left( \mathcal{L}_{\zeta }g\right) \left(
X,Y\right) +2hf\dot{f}g\left( X,Y\right)  \label{e1}
\end{equation}%
where $\bar{\zeta}=h\partial _{t}+\zeta ,$ $\bar{X}=x\partial _{t}+X$ and $%
\bar{Y}=y\partial _{t}+Y$.
\end{proposition}

An important consequence of this proposition is the following.

\begin{theorem}
Let $\bar{M}=I\times _{f}M$ \ be a generalized Robertson-Walker spacetime
equipped with the metric tensor $\bar{g}=-dt^{2}\oplus f^{2}g$. Then,

\begin{enumerate}
\item a time-like vector field $\bar{\zeta}=h\partial _{t}\in \mathfrak{X}(%
\bar{M})$ is a conformal vector field on $\bar{M}$ if and only if $h=af$
where $a$ is constant. Moreover, the conformal factor is $2\dot{h}.$

\item a space like $\bar{\zeta}=\zeta \in \mathfrak{X}(\bar{M})$ is a
Killing vector field on $\bar{M}$ if and only if $\zeta \in \mathfrak{X}(M)$
is Killing vector field on $M$.

\item a space like $\bar{\zeta}=\zeta \in \mathfrak{X}(\bar{M})$ is a Matter
collineation on $\bar{M}$ if $\zeta \in \mathfrak{X}(M)$ is a Killing vector
field on $M$.
\end{enumerate}
\end{theorem}

\begin{proof}
Suppose that $h=af$. If $h=0$, then $a=0$ and the result is obvious. Now, we
assume that $h\neq 0$. Using equation (\ref{e1}) from the above proposition
we get that%
\begin{eqnarray*}
\left( \mathcal{\bar{L}}_{\bar{\zeta}}\bar{g}\right) \left( \bar{X},\bar{Y}%
\right) &=&-2\dot{h}xy+f^{2}\left( \mathcal{L}_{\zeta }g\right) \left(
X,Y\right) +2hf\dot{f}g\left( X,Y\right) \\
&=&2\dot{h}\bar{g}\left( \bar{X},\bar{Y}\right)
\end{eqnarray*}%
i.e.$\bar{\zeta}=h\partial _{t}$ is a conformal vector field with conformal
factor $\rho =2\dot{h}$. Conversely, suppose that $\bar{\zeta}=h\partial
_{t}\in \mathfrak{X}(\bar{M})$ is a conformal vector field with factor $\rho
,$ then%
\begin{equation*}
\left( \mathcal{\bar{L}}_{\bar{\zeta}}\bar{g}\right) \left( \bar{X},\bar{Y}%
\right) =\rho \bar{g}\left( \bar{X},\bar{Y}\right)
\end{equation*}%
for any vector fields $\bar{X},\bar{Y}\in \mathfrak{X}(\bar{M})$. Now, by
equation (\ref{e1}), we get that%
\begin{equation*}
-\rho xy+\rho f^{2}g\left( X,Y\right) =-2\dot{h}xy+2hf\dot{f}g\left(
X,Y\right)
\end{equation*}%
Let $X=Y=0$, we get that $\rho =2\dot{h}$. Now, let us put $x=y=0$. This
yields $\rho f=2h\dot{f}$. These two differential equations imply that $h=0$
or%
\begin{equation*}
\dot{h}f=h\dot{f}
\end{equation*}%
and so $h=af$ where $a$ is constant. The second and third assertions are
direct from equation (\ref{e1}) when $h=0$.
\end{proof}

It is well-known that if $X$ and $Y$ are conformal vector fields on $\bar{M}$
with $X=\mu Y$ for some smooth function $\mu $, then $\mu $ is constant.
Thus the above result represents a good characterization to both time-like
conformal vector fields and space-like Killing vector fields on $\bar{M}$.

\begin{theorem}
A vector field\ $\bar{\zeta}=h\partial _{t}+\zeta $ on $\bar{M}=I\times
_{f}M $ is conformal if and only if $\zeta $ is a conformal vector field on $%
M$ with factor $\rho =2\left( \dot{h}-\frac{h\dot{f}}{f}\right) $. Moreover,
the conformal factor of $\bar{\zeta}$ is $2\dot{h}$.
\end{theorem}

\begin{proof}
Let $\bar{\zeta}=h\partial _{t}+\zeta $ be a conformal vector field on $\bar{%
M}=I\times _{f}M$ with factor $\bar{\rho}$, then%
\begin{equation*}
-\bar{\rho}xy+\bar{\rho}f^{2}g\left( X,Y\right) =-2\dot{h}xy+f^{2}\left( 
\mathcal{L}_{\zeta }g\right) \left( X,Y\right) +2hf\dot{f}g\left( X,Y\right)
\end{equation*}%
for any vector fields $\bar{X}=x\partial _{t}+X$ and $\bar{Y}=y\partial
_{t}+Y$. This equation yields $\bar{\rho}=2\dot{h}$ and hence%
\begin{equation*}
\left( \mathcal{L}_{\zeta }g\right) \left( X,Y\right) =\left( \bar{\rho}-%
\frac{2h\dot{f}}{f}\right) g\left( X,Y\right)
\end{equation*}%
i.e. $\zeta $ is a conformal vector field on $M$\ with conformal factor $%
\rho =2\dot{h}-\frac{2h\dot{f}}{f}$. The converse is direct.
\end{proof}

These two results reveal that the dimension of the conformal algebra $%
C\left( \bar{M}\right) $ of $\bar{M}$ is at least $r+1$ where $r$ is the
dimension of $C\left( M\right) $. Suppose that $M$ is maximally symmetric,
then%
\begin{equation*}
\frac{n^{2}+3n+4}{2}\leq \dim C\left( \bar{M}\right) \leq \frac{\left(
n+2\right) \left( n+3\right) }{2}
\end{equation*}

Let $\bar{\zeta}=h\partial _{t}+\zeta $ be a Killing vector field on a
generalized Robertson-Walker spacetime $\bar{M}$ and let $r=\left(
1/2\right) \bar{g}\left( \bar{\zeta},\bar{\zeta}\right) $. Then%
\begin{equation*}
\bar{g}\left( \bar{\nabla}r,\bar{X}\right) =-\bar{g}\left( \bar{D}_{\bar{%
\zeta}}\bar{\zeta},\bar{X}\right)
\end{equation*}%
for any vector field $\bar{X}\in \mathfrak{X}\left( \bar{M}\right) $ i.e. $%
\bar{\nabla}r=-\bar{D}_{\bar{\zeta}}\bar{\zeta}$. Thus the Hessian
definition with some computations yield%
\begin{equation*}
\bar{H}^{r}\left( \bar{X},\bar{X}\right) =-\mathrm{\bar{R}}\left( \bar{\zeta}%
,\bar{X},\bar{\zeta},\bar{X}\right) +\bar{g}\left( \bar{D}_{\bar{X}}\bar{%
\zeta},\bar{D}_{\bar{X}}\bar{\zeta}\right)
\end{equation*}%
Taking the trace of both sides imply%
\begin{equation*}
\bar{\Delta}r=-\mathrm{\bar{R}ic}\left( \bar{\zeta},\bar{\zeta}\right) +\bar{%
g}\left( \bar{D}\bar{\zeta},\bar{D}\bar{\zeta}\right)
\end{equation*}%
Assume that $\dot{f}=\dot{h}=0$. Then%
\begin{equation*}
\bar{\Delta}r=-\mathrm{Ric}\left( \zeta ,\zeta \right) +f^{2}g\left( D\zeta
,D\zeta \right)
\end{equation*}

\begin{theorem}
Let $\bar{\zeta}=h\partial _{t}+\zeta $ be a Killing vector field on a
generalized Robertson-Walker spacetime $\bar{M}$ and let $r=\frac{1}{2}\bar{g%
}\left( \bar{\zeta},\bar{\zeta}\right) $. Then%
\begin{equation*}
\bar{\Delta}r=-\mathrm{\bar{R}ic}\left( \bar{\zeta},\bar{\zeta}\right) +\bar{%
g}\left( \bar{D}\bar{\zeta},\bar{D}\bar{\zeta}\right)
\end{equation*}%
Moreover, if $\dot{f}=\dot{h}=0$, then%
\begin{equation*}
\bar{\Delta}r=-\mathrm{Ric}\left( \zeta ,\zeta \right) +f^{2}g\left( D\zeta
,D\zeta \right)
\end{equation*}
\end{theorem}

The following result is an analogue of a similar result in Riemannian
manifolds.

\begin{theorem}
A Killing vector field $\bar{\zeta}=h\partial _{t}+\zeta $ on a generalized
Robertson-Walker spacetime $\bar{M}=I\times _{f}M$ equipped with the metric
tensor $\bar{g}=-dt^{2}\oplus f^{2}g$ has a constant length if and only if\ $%
\zeta $ satisfies%
\begin{equation}
D_{\zeta }\zeta +\dfrac{2h\dot{f}}{f}\zeta =0\text{ and }h\dot{h}-f\dot{f}%
g\left( \zeta ,\zeta \right) =0  \label{e8}
\end{equation}
\end{theorem}

\begin{proof}
Suppose that $\bar{\zeta}$ is a Killing vector field. Then%
\begin{equation*}
\bar{g}\left( \bar{D}_{\bar{X}}\bar{\zeta},\bar{Y}\right) +\bar{g}\left( 
\bar{X},\bar{D}_{\bar{Y}}\bar{\zeta}\right) =0
\end{equation*}%
for any vector fields $\bar{X},\bar{Y}\in \mathfrak{X}\left( \bar{M}\right) $%
. Let $\bar{X}=\bar{\zeta}$ in the above equation, then%
\begin{equation*}
\bar{g}\left( \bar{D}_{\bar{\zeta}}\bar{\zeta},\bar{Y}\right) =-\bar{g}%
\left( \bar{\zeta},\bar{D}_{\bar{Y}}\bar{\zeta}\right)
\end{equation*}%
But Equations (\ref{Connection}) and (\ref{e8}) yield $\bar{D}_{\bar{\zeta}}%
\bar{\zeta}=0$ and so%
\begin{equation*}
\bar{g}\left( \bar{\zeta},\bar{D}_{\bar{Y}}\bar{\zeta}\right) =0
\end{equation*}%
i.e. $\bar{g}\left( \bar{\zeta},\bar{\zeta}\right) $ is constant and so $%
\bar{\zeta}$ has a constant length. Conversely, if $\bar{\zeta}$ has
constant length, then%
\begin{equation*}
\bar{g}\left( \bar{\zeta},\bar{D}_{\bar{Y}}\bar{\zeta}\right) =-\bar{g}%
\left( \bar{D}_{\bar{\zeta}}\bar{\zeta},\bar{Y}\right) =0
\end{equation*}%
for any vector field $\bar{Y}\in \mathfrak{X}\left( \bar{M}\right) $ i.e. $%
\bar{D}_{\bar{\zeta}}\bar{\zeta}=0$. Now, we can use Equation (\ref%
{Connection}) to get the result.
\end{proof}

\begin{corollary}
Let $\bar{\zeta}=h\partial _{t}+\zeta $ be a Killing vector field $of$
constant length on a generalized Robertson-Walker spacetime $\bar{M}=I\times
_{f}M$ equipped with the metric tensor $\bar{g}=-dt^{2}\oplus f^{2}g$. Then
the flow lines of $\zeta $ are geodesics on $M$ if and only if $f$ is
constant or $h=0$.
\end{corollary}

\begin{theorem}
Let $\bar{\zeta}=h\partial _{t}+\zeta $ be a Killing vector field on a
generalized Robertson-Walker spacetime $\bar{M}=I\times _{f}M$ equipped with
the metric tensor $\bar{g}=-dt^{2}\oplus f^{2}g$ and $\alpha (s)$, $s\in 
\mathbb{R}
$, be a geodesic on $(\bar{M},\bar{g})$ with tangent vector field $\bar{X}%
=x\partial _{t}+X$. Assume that $\dot{f}=0$. Then $h$ is constant and $\zeta
\in \mathfrak{X}\left( M\right) $ is a Jacobi vector field along the
integral curves of $X$.
\end{theorem}

\begin{theorem}
Let $\bar{\zeta}=h\partial _{t}+\zeta $ be a conformal vector field along a
curve $\alpha (s)$ with unit tangent vector $\bar{V}=v\partial _{t}+V$ on a
generalized Robertson-Walker spacetime $\bar{M}=I\times _{f}M$ equipped with
the metric tensor $\bar{g}=-dt^{2}\oplus f^{2}g$. Then the conformal factor $%
\rho $ of $\bar{\zeta}$ is given by%
\begin{equation*}
\rho =2\left[ v^{2}\dot{h}+hf\dot{f}\left\Vert V\right\Vert
^{2}+f^{2}g\left( D_{V}\zeta ,V\right) \right]
\end{equation*}
\end{theorem}

\begin{proof}
Let $\bar{\zeta}$ be a conformal vector field with conformal factor $\rho $.
Then%
\begin{equation*}
\left( \mathcal{\bar{L}}_{\bar{\zeta}}\bar{g}\right) (\bar{X},\bar{Y})=\rho 
\bar{g}\left( \bar{X},\bar{Y}\right) 
\end{equation*}%
Let us put $\bar{X}=\bar{Y}=\bar{V}$, then the conformal factor $\rho $ is
given by%
\begin{equation*}
\rho =2\bar{g}(\bar{D}_{\bar{V}}\bar{\zeta},\bar{V})
\end{equation*}%
Thus%
\begin{eqnarray*}
\rho  &=&2\bar{g}(\bar{D}_{\bar{V}}\bar{\zeta},\bar{V}) \\
&=&2g\left( v\dot{h}\partial _{t}+\frac{v\dot{f}}{f}\zeta +\frac{h\dot{f}}{f}%
V+D_{V}\zeta -f\dot{f}g(\zeta ,V)\partial _{t},V\right)  \\
&=&2\left[ v^{2}\dot{h}+hf\dot{f}\left\Vert V\right\Vert ^{2}+f^{2}g\left(
D_{V}\zeta ,V\right) \right] 
\end{eqnarray*}
\end{proof}

In the sequential, we study the structure of concircular vector fields on
generalized Robertson-Walker spacetimes.

\begin{theorem}
\label{concircular1}Let $\bar{\zeta}=h\partial _{t}+\zeta $ be a vector
field on a generalized Robertson-Walker spacetime $\bar{M}=I\times _{f}M$
equipped with the metric tensor $\bar{g}=-dt^{2}\oplus f^{2}g$. Then $\bar{%
\zeta}$ is a concircular vector field on $\bar{M}$ if and only if one of the
following conditions holds:

\begin{enumerate}
\item $h=af$ and $\zeta =0$, or

\item $\zeta $ is a concircular vector field on $M$ with factor $\rho =\dot{h%
}$ and $f$ is constant.
\end{enumerate}
\end{theorem}

\begin{proof}
Let $\bar{X}=x\partial _{t}+X$ be any vector field on $\bar{M}$. Then for
any scalar function $\rho $ we get that%
\begin{equation*}
\bar{D}_{\bar{X}}\bar{\zeta}-\rho \bar{X}=\left( x\dot{h}-f\dot{f}g\left(
X,\zeta \right) -x\rho \right) \partial _{t}+\frac{x\dot{f}}{f}\zeta +\frac{h%
\dot{f}}{f}X+D_{X}\zeta -\rho X
\end{equation*}

Suppose that $\bar{\zeta}$ is concircular on $\bar{M}$, then%
\begin{eqnarray*}
x\dot{h}-f\dot{f}g\left( X,\zeta \right) -x\rho &=&0 \\
\frac{x\dot{f}}{f}\zeta +\frac{h\dot{f}}{f}X+D_{X}\zeta -\rho X &=&0
\end{eqnarray*}%
If $\dot{f}$ does not vanish, then $g\left( X,\zeta \right) =0$ for all $X$
i.e. $\zeta =0$. Thus%
\begin{eqnarray*}
x\left( \dot{h}-\rho \right) &=&0 \\
\left( \frac{h\dot{f}}{f}-\rho \right) X &=&0
\end{eqnarray*}%
and therefore $\dot{h}=\frac{h\dot{f}}{f}$ i.e. $h=af$ for some constant.

However, if $\dot{f}=0$, then we get that%
\begin{eqnarray*}
x\left( \dot{h}-\rho \right) &=&0 \\
D_{X}\zeta -\rho X &=&0
\end{eqnarray*}%
i.e. $\zeta $ is a concircular vector field on $M$ with factor $\rho =\dot{h}
$.

Conversely, we have%
\begin{equation*}
\bar{D}_{\bar{X}}\bar{\zeta}=\left[ x\dot{h}-f\dot{f}g\left( X,\zeta \right) %
\right] \partial _{t}+D_{X}\zeta +\left( x\frac{\dot{f}}{f}\right) \zeta
+\left( h\frac{\dot{f}}{f}\right) X
\end{equation*}%
Finally, any one of the above conditions implies that $\bar{D}_{\bar{X}}\bar{%
\zeta}=\dot{h}\bar{X}$ and consequently $\bar{\zeta}$ is concircular on $%
\bar{M}$.
\end{proof}

\begin{theorem}
\label{concircular2}Let $\bar{\zeta}=h\partial _{t}+\zeta $ be a concircular
vector field on a generalized Robertson-Walker spacetime $\bar{M}=I\times
_{f}M$ equipped with the metric tensor $\bar{g}=-dt^{2}\oplus f^{2}g$. Then 
\begin{equation*}
{\mathrm{Ric}}\left( \zeta ,\zeta \right) =0\text{ \ \ \ and\ \ \ }\kappa
\left( X,\zeta \right) =0
\end{equation*}%
for any vector field $X\in \mathfrak{X}(M)$.
\end{theorem}

\begin{proof}
Using the above theorem we get that $\zeta $ is zero or concircular with
factor $\dot{h}$. If $\zeta =0$, then {$\mathrm{Ric}$}$\left( \zeta ,\zeta
\right) =0$ and $\kappa \left( X,\zeta \right) =0$. Now suppose that $\zeta $
is concircular on $M$ i.e.%
\begin{equation*}
D_{X}\zeta =\dot{h}X
\end{equation*}%
for any vector field $X\in \mathfrak{X}(M)$. Let $e\in \mathfrak{X}(M)$ be a
unit vector field, then%
\begin{eqnarray*}
R\left( \zeta ,e,\zeta ,e\right) &=&g\left( -D_{\zeta }D_{e}\zeta
+D_{e}D_{\zeta }\zeta +D_{\left[ \zeta ,e\right] }\zeta ,e\right) \\
&=&g\left( -D_{\zeta }\left( \dot{h}e\right) +D_{e}\left( \dot{h}\zeta
\right) +\dot{h}\left[ \zeta ,e\right] ,e\right) \\
&=&\dot{h}g\left( -D_{\zeta }e+D_{e}\zeta +\left[ \zeta ,e\right] ,e\right)
\\
&=&0
\end{eqnarray*}%
Thus {$\mathrm{Ric}$}$\left( \zeta ,\zeta \right) =0$ and $\kappa \left(
X,\zeta \right) =0$.
\end{proof}

Now, let us turn to curvature collineations. First, we present the following
important proposition.

\begin{proposition}
Let $\bar{\zeta}=h\partial _{t}+\zeta $ be a vector field on a generalized
Robertson-Walker spacetime $\bar{M}=I\times _{f}M$ equipped with the metric
tensor $\bar{g}=-dt^{2}\oplus f^{2}g$. Then%
\begin{eqnarray*}
\left( \mathcal{\bar{L}}_{h\partial _{t}}\mathrm{\bar{R}}\right) \left(
\partial _{t},\partial _{t},\partial _{t},\partial _{t}\right) &=&\left( 
\mathcal{\bar{L}}_{h\partial _{t}}\mathrm{\bar{R}}\right) \left( X,\partial
_{t},\partial _{t},\partial _{t}\right) =0 \\
\left( \mathcal{\bar{L}}_{h\partial _{t}}\mathrm{\bar{R}}\right) \left(
X,Y,\partial _{t},\partial _{t}\right) &=&\left( \mathcal{\bar{L}}%
_{h\partial _{t}}\mathrm{\bar{R}}\right) \left( X,Y,Z,\partial _{t}\right) =0
\\
\left( \mathcal{\bar{L}}_{h\partial _{t}}\mathrm{\bar{R}}\right) \left(
X,Y,Z,W\right) &=&0
\end{eqnarray*}%
\begin{equation*}
\left( \mathcal{\bar{L}}_{h\partial _{t}}\mathrm{\bar{R}}\right) \left(
Y,\partial _{t},\partial _{t},W\right) =-\left[ h\dot{f}\ddot{f}+hf\dddot{f}%
+2\dot{h}f\ddot{f}\right] g\left( Y,W\right)
\end{equation*}%
\begin{eqnarray*}
\left( \mathcal{\bar{L}}_{\zeta }\mathrm{\bar{R}}\right) \left( \partial
_{t},\partial _{t},\partial _{t},\partial _{t}\right) &=&\left( \mathcal{%
\bar{L}}_{\zeta }\mathrm{\bar{R}}\right) \left( X,\partial _{t},\partial
_{t},\partial _{t}\right) =0 \\
\left( \mathcal{\bar{L}}_{\zeta }\mathrm{\bar{R}}\right) \left( X,Y,\partial
_{t},\partial _{t}\right) &=&\left( \mathcal{\bar{L}}_{\zeta }\mathrm{\bar{R}%
}\right) \left( X,Y,Z,\partial _{t}\right) =0
\end{eqnarray*}%
\begin{equation*}
\left( \mathcal{\bar{L}}_{\zeta }\mathrm{\bar{R}}\right) \left( Y,\partial
_{t},\partial _{t},W\right) =-f\ddot{f}\left( \mathcal{L}_{\zeta }g\right)
\left( Y,W\right)
\end{equation*}%
\begin{eqnarray*}
\left( \mathcal{\bar{L}}_{\zeta }\mathrm{\bar{R}}\right) \left(
X,Y,Z,W\right) &=&f^{2}\left( \mathcal{L}_{\zeta }R\right) \left(
X,Y,Z,W\right) \\
&&+f^{2}\dot{f}^{2}\left[ \left( \mathcal{L}_{\zeta }g\right) \left(
X,Z\right) g\left( Y,W\right) +g\left( X,Z\right) \left( \mathcal{L}_{\zeta
}g\right) \left( Y,W\right) \right] \\
&&-f^{2}\dot{f}^{2}\left[ \left( \mathcal{L}_{\zeta }g\right) \left(
Y,Z\right) g\left( X,W\right) +g\left( Y,Z\right) \left( \mathcal{L}_{\zeta
}g\right) \left( X,W\right) \right]
\end{eqnarray*}%
where $X,Y,Z,W\in \mathfrak{X}(M)$.
\end{proposition}

The following theorem represents a characterization of curvature
collineations on generalized Robertson-Walker spacetimes.

\begin{theorem}
Let $\bar{M}=I\times _{f}M$ be a generalized Robertson-Walker spacetime
equipped with the metric tensor $\bar{g}=-dt^{2}\oplus f^{2}g$. Then,

\begin{enumerate}
\item $\bar{\zeta}=h\partial _{t}$ is a curvature collineation on $\bar{M}$
if and only if $h\dot{f}\ddot{f}+hf\dddot{f}+2\dot{h}f\ddot{f}=0$.

\item $\bar{\zeta}=\zeta $ is a curvature collineation on $\bar{M}$ if $%
\zeta $ is a Killing vector field on $M$.
\end{enumerate}
\end{theorem}

\begin{theorem}
Let $\bar{\zeta}=h\partial _{t}+\zeta $ be a curvature collineation on a
generalized Robertson-Walker spacetime $\bar{M}=I\times _{f}M$ equipped with
the metric tensor $\bar{g}=-dt^{2}\oplus f^{2}g$. Assume that $\Delta f\neq
0 $. Then $\zeta $ is a Killing vector field on $M$.
\end{theorem}

Now we will do the same job for Ricci collineations on generalized
Robertson-Walker spacetimes

\begin{proposition}
Let $\bar{\zeta}=h\partial _{t}+\zeta \in \mathfrak{X}\left( \bar{M}\right) $%
, then%
\begin{eqnarray*}
\left( \mathcal{\bar{L}}_{\bar{\zeta}}\mathrm{\bar{R}ic}\right) \left( \bar{X%
},\bar{Y}\right) &=&\frac{nxy}{f^{2}}\left( hf\dddot{f}-h\dot{f}\ddot{f}+2%
\dot{h}f\ddot{f}\right) +h\left( f\dddot{f}+\left( 2n-1\right) \dot{f}\ddot{f%
}\right) g\left( X,Y\right) \\
&&+\left( \mathcal{L}_{\zeta }\mathrm{Ric}\right) \left( X,Y\right)
-f^{\diamond }\left( \mathcal{L}_{\zeta }g\right) \left( X,Y\right)
\end{eqnarray*}%
for any $X,Y\in \mathfrak{X}(M)$
\end{proposition}

The following results are immediate consequences of the above proposition.

\begin{theorem}
Let $\bar{\zeta}=h\partial _{t}\in \mathfrak{X}\left( \bar{M}\right) $ be a
vector field on a generalized Robertson-Walker spacetime $\bar{M}=I\times
_{f}M$. Assume that $H^{f}=0$. Then, $\bar{\zeta}$ is a Ricci collineation
on $\bar{M}$.
\end{theorem}

Now, we consider the converse of the above result.

\begin{theorem}
Let $\bar{\zeta}=h\partial _{t}\in \mathfrak{X}\left( \bar{M}\right) $ be a
Ricci collineation on a generalized Robertson-Walker spacetime $\bar{M}%
=I\times _{f}M$. Then one of the following conditions holds

\begin{enumerate}
\item $H^{f}=0$, or

\item $h=af^{n}$ for some constant $a$.
\end{enumerate}
\end{theorem}

\begin{proof}
Suppose that $\bar{\zeta}$ is a Ricci collineation on $\bar{M}$. Then%
\begin{eqnarray*}
\left( \mathcal{\bar{L}}_{\bar{\zeta}}\mathrm{\bar{R}ic}\right) \left(
x\partial _{t},y\partial _{t}\right) &=&\frac{nxy}{f^{2}}\left( hf\dddot{f}-h%
\dot{f}\ddot{f}+2\dot{h}f\ddot{f}\right) =0 \\
\left( \mathcal{\bar{L}}_{\bar{\zeta}}\mathrm{\bar{R}ic}\right) \left(
X,Y\right) &=&h\left( f\dddot{f}+\left( 2n-1\right) \dot{f}\ddot{f}\right)
g\left( X,Y\right) =0
\end{eqnarray*}%
The second equation yields%
\begin{equation*}
f\dddot{f}=-\left( 2n-1\right) \dot{f}\ddot{f}
\end{equation*}%
and so the first equation implies%
\begin{equation*}
2\ddot{f}\left[ nh\dot{f}-\dot{h}f\right] =0
\end{equation*}%
which proves the result.
\end{proof}

\begin{theorem}
Let $\bar{\zeta}=\zeta \in \mathfrak{X}\left( \bar{M}\right) $ be a vector
field on a generalized Robertson-Walker spacetime $\bar{M}=I\times _{f}M$.
Assume that $f^{\diamond }=0$. Then, $\bar{\zeta}$ is a Ricci collineation
on $\bar{M}$ if and only if $\zeta $ is a Ricci collineation on $M$.
\end{theorem}

A vector field $\zeta \in \mathfrak{X}\left( M\right) $ is called a
conformal Ricci collineation if%
\begin{equation*}
\left( \mathcal{L}_{\zeta }\mathrm{Ric}\right) \left( X,Y\right) =\rho
g\left( X,Y\right)
\end{equation*}%
for some smooth function $\rho $ on $M$.

\begin{theorem}
Let $\bar{\zeta}=h\partial _{t}+\zeta \in \mathfrak{X}\left( \bar{M}\right) $
be a Ricci collineation on a generalized Robertson-Walker spacetime $\bar{M}%
=I\times _{f}M$. Then, $\zeta $ is a conformal Ricci collineation on $M$ if
and only if $\zeta $ is a conformal vector field on $M$.
\end{theorem}

\section{Ricci Soliton}

A smooth vector field $\zeta $ on a Riemannian manifold $\left( M,g\right) $
is said to define a Ricci soliton $\left( M,g,\zeta ,\lambda \right) $ if it
satisfies the soliton equation%
\begin{equation}
\frac{1}{2}\mathcal{L}_{\zeta }g+\mathrm{Ric}=\lambda g
\label{Ricci soliton}
\end{equation}%
where $\mathcal{L}_{\zeta }$ is the Lie-derivative with respect to $\zeta $, 
\textrm{$Ric$} is the Ricci tensor and $\lambda $ is a constant. If $\zeta =%
\mathrm{grad}u$, for a smooth function $u$ on $M$, the Ricci soliton $\left(
M,g,\zeta ,\lambda \right) =\left( M,g,u,\lambda \right) $ is called a
gradient Ricci soliton and the function $u$ is called the potential
function. The study of Ricci solitons was first introduced by Hamilton as
fixed or stationary points of the Ricci flow in the space of the metrics on $%
M$ modulo diffeomorphism and scaling. Gradient Ricci solitons are natural
generalizations of Einstein manifolds\cite{Barros:2012, Barros:2013,
Bernstein:2015, Fernandez:2011, Peterson:2009, Munteanu:2013}.

Let us take the Lie derivative of both sides of Equation (\ref{Ricci soliton}%
) in direction of $\zeta $, then we have%
\begin{equation*}
\frac{1}{2}\mathcal{L}_{\zeta }\mathcal{L}_{\zeta }g+\mathcal{L}_{\zeta }%
\mathrm{Ric}=\lambda \mathcal{L}_{\zeta }g
\end{equation*}%
A vector field $\zeta $ is called $2-$Killing if $\mathcal{L}_{\zeta }%
\mathcal{L}_{\zeta }g=0$. Thus the above equation reveals the following
result.

\begin{proposition}
Let $\left( M,g,\zeta ,\lambda \right) $ be a Ricci soliton where $\zeta $
is a $2-$Killing vector field. Then,

\begin{enumerate}
\item $\zeta $ is Killing if and only if $\left( M,g\right) $ is Einstein.
Moreover, the Einstein factor is $\lambda $.

\item $\zeta $ is Killing if and only if $\zeta $ is a Ricci collineation.
\end{enumerate}
\end{proposition}

Now, we consider a Ricci soliton structure on a generalized Robertson-Walker
spacetime.

\begin{theorem}
Let $\left( \bar{M},\bar{g},\bar{\zeta},\lambda \right) $ be a Ricci soliton
where $\bar{M}=I\times _{f}M$ is a generalized Robertson-Walker spacetime
and $\bar{\zeta}=h\partial _{t}+\zeta \in \mathfrak{X}\left( \bar{M}\right) $%
. Then

\begin{enumerate}
\item $h\partial _{t}$ is conformal on $I$ with factor $2\left( \lambda +n%
\frac{\ddot{f}}{f}\right) $, and

\item $\left( M,g,f^{2}\zeta ,\lambda f^{2}+f^{\diamond }-hf\dot{f}\right) $
is a Ricci soliton whenever $\lambda f^{2}+f^{\diamond }-hf\dot{f}$ is
constant.
\end{enumerate}
\end{theorem}

\begin{proof}
Let $\left( \bar{M},\bar{g},\bar{\zeta},\lambda \right) $ be a Ricci
soliton, then%
\begin{equation*}
\frac{1}{2}\left( \mathcal{L}_{\bar{\zeta}}\bar{g}\right) \left( \bar{X},%
\bar{Y}\right) +\mathrm{\bar{R}ic}\left( \bar{X},\bar{Y}\right) =\lambda 
\bar{g}\left( \bar{X},\bar{Y}\right)
\end{equation*}%
where $\bar{X}=x\partial _{t}+X$ and $\bar{Y}=y\partial _{t}+Y$ are vector
fields on $\bar{M}$. Let $X=Y=0$, then%
\begin{eqnarray*}
\frac{1}{2}\left( \mathcal{L}_{\bar{\zeta}}\bar{g}\right) \left( x\partial
_{t},y\partial _{t}\right) +\mathrm{\bar{R}ic}\left( x\partial
_{t},y\partial _{t}\right) &=&\lambda \bar{g}\left( x\partial _{t},y\partial
_{t}\right) \\
\frac{1}{2}\left( \mathcal{L}_{h\partial _{t}}^{I}g_{I}\right) \left(
x\partial _{t},y\partial _{t}\right) +{\mathrm{Ric}}^{I}\left( x\partial
_{t},y\partial _{t}\right) -\frac{n}{f}H^{f}\left( x\partial _{t},y\partial
_{t}\right) &=&\lambda g_{I}\left( x\partial _{t},y\partial _{t}\right) \\
\frac{1}{2}\left( \mathcal{L}_{h\partial _{t}}^{I}g_{I}\right) (x\partial
_{t},y\partial _{t}) &=&\left( \lambda +\frac{n}{f}\ddot{f}\right)
g_{I}\left( x\partial _{t},y\partial _{t}\right)
\end{eqnarray*}%
Thus, $h\partial _{t}$ is a conformal vector field on $I$ with factor $%
2\left( \lambda +n\frac{\ddot{f}}{f}\right) $. Now, let $x=y=0$, then%
\begin{equation*}
\frac{1}{2}\left( f^{2}\left( \mathcal{L}_{\zeta }g\right) (X,Y)+2hf\dot{f}%
g(X,Y)\right) +{\mathrm{Ric}}\left( X,Y\right) -f^{\diamond }g\left(
X,Y\right) =\lambda f^{2}g\left( X,Y\right)
\end{equation*}%
where $f^{\diamond }=-f\ddot{f}-\left( n-1\right) \dot{f}^{2}$. Simply, we
may rewrite it as follows,%
\begin{equation*}
\frac{1}{2}f^{2}\left( \mathcal{L}_{\zeta }g\right) (X,Y)+{\mathrm{Ric}}%
\left( X,Y\right) =\left( f^{\diamond }-hf\dot{f}+\lambda f^{2}\right)
g\left( X,Y\right)
\end{equation*}%
Thus $\left( M,g,f^{2}\zeta ,\mu \right) $ is a Ricci soliton where $\mu
=\lambda f^{2}+f^{\diamond }-hf\dot{f}$.
\end{proof}

\begin{theorem}
Let $\left( \bar{M},\bar{g},\bar{\zeta},\lambda \right) $ be a Ricci soliton
where $\overline{M}=I\times _{f}M$ is a generalized Robertson-Walker
spacetime and $\bar{\zeta}=h\partial _{t}+\zeta \in \mathfrak{X}\left( \bar{M%
}\right) $ be a conformal vector field on $\bar{M}$. Then $\left( M,g\right) 
$ is an Einstein manifold with factor 
\begin{equation*}
\mu =-\left[ \left( n+1\right) f\ddot{f}+\left( n-1\right) \dot{f}^{2}\right]
\end{equation*}%
Moreover, the conformal factor is $\lambda +\frac{n\ddot{f}}{f}$.
\end{theorem}

\begin{proof}
Let $\left( \bar{M},\bar{g},\bar{\zeta},\lambda \right) $ be a Ricci soliton
where $\overline{M}=I\times _{f}M$ is a generalized Robertson-Walker
spacetime and $\bar{\zeta}=h\partial _{t}+\zeta \in \mathfrak{X}\left( \bar{M%
}\right) $ be a conformal vector field on $\overline{M}$. Then%
\begin{equation*}
\mathrm{\bar{R}ic}\left( \bar{X},\bar{Y}\right) =\left( \lambda -\rho
\right) \bar{g}\left( \bar{X},\bar{Y}\right)
\end{equation*}%
Then%
\begin{eqnarray*}
&&{\mathrm{Ric}}^{I}\left( x\partial _{t},y\partial _{t}\right) -\frac{n}{f}%
H^{f}\left( x\partial _{t},y\partial _{t}\right) +{\mathrm{Ric}}\left(
X,Y\right) -f^{\diamond }g\left( X,Y\right) \\
&=&\left( \lambda -\rho \right) \left( g_{I}\left( x\partial _{t},y\partial
_{t}\right) +f^{2}g\left( X,Y\right) \right)
\end{eqnarray*}%
Let $X=Y=0$, then%
\begin{equation*}
\mathrm{\bar{R}ic}\left( x\partial _{t},y\partial _{t}\right) =\left(
\lambda -\rho \right) \bar{g}\left( x\partial _{t},y\partial _{t}\right)
\end{equation*}%
Using Proposition (\ref{Ricci}), we get that%
\begin{eqnarray*}
{\mathrm{Ric}}^{I}\left( x\partial _{t},y\partial _{t}\right) -\frac{n}{f}%
H^{f}\left( x\partial _{t},y\partial _{t}\right) &=&\left( \lambda -\rho
\right) g_{I}\left( x\partial _{t},y\partial _{t}\right) \\
-\frac{n}{f}\ddot{f}g_{I}\left( x\partial _{t},y\partial _{t}\right)
&=&\left( \lambda -\rho \right) g_{I}\left( x\partial _{t},y\partial
_{t}\right)
\end{eqnarray*}%
Thus%
\begin{equation}
\lambda -\rho =-\frac{n\ddot{f}}{f}  \label{E10}
\end{equation}%
Now, we let $x=y=0$, then%
\begin{eqnarray*}
\mathrm{\bar{R}ic}\left( X,Y\right) &=&\left( \lambda -\rho \right) \bar{g}%
\left( X,Y\right) \\
{\mathrm{Ric}}\left( X,Y\right) -f^{\diamond }g\left( X,Y\right) &=&\left(
\lambda -\rho \right) f^{2}g\left( X,Y\right) \\
{\mathrm{Ric}}\left( X,Y\right) &=&\left[ \left( \lambda -\rho \right)
f^{2}+f^{\diamond }\right] g\left( X,Y\right)
\end{eqnarray*}%
where 
\begin{equation}
f^{\diamond }=-f\ddot{f}-\left( n-1\right) \dot{f}^{2}  \label{E11}
\end{equation}%
By using equations (\ref{E10}) and (\ref{E11}), we get that%
\begin{equation}
{\mathrm{Ric}}\left( X,Y\right) =-\left[ \left( n+1\right) f\ddot{f}+\left(
n-1\right) \dot{f}^{2}\right] g\left( X,Y\right)  \notag
\end{equation}%
Then $\left( M,g\right) $ is an Einstein manifold with factor $\mu =-\left[
\left( n+1\right) f\ddot{f}+\left( n-1\right) \dot{f}^{2}\right] $.
\end{proof}

\begin{corollary}
Let $\left( \bar{M},\bar{g},\bar{\zeta},\lambda \right) $ be a Ricci soliton
where $\bar{M}=I\times _{f}M$ is a generalized Robertson-Walker spacetime
and $\bar{\zeta}=h\partial _{t}+\zeta \in \mathfrak{X}\left( \bar{M}\right) $
be a conformal vector field on $\bar{M}$. Then $\left( M,g\right) $ is Ricci
flat if $f$ is constant.
\end{corollary}

\begin{corollary}
Let $\left( \bar{M},\bar{g},\bar{\zeta},\lambda \right) $ be a Ricci soliton
where $\bar{M}=I\times _{f}M$ is a generalized Robertson-Walker spacetime
and $\bar{\zeta}=h\partial _{t}+\zeta \in \mathfrak{X}\left( \bar{M}\right) $
be a Killing vector field on $\bar{M}$. Then $\lambda =-\frac{n\ddot{f}}{f}$.
\end{corollary}

\begin{theorem}
Let $\left( \bar{M},\bar{g},\bar{\zeta},\lambda \right) $ be a Ricci soliton
where $\bar{M}=I\times _{f}M$ is a generalized Robertson-Walker spacetime
and $\bar{\zeta}=h\partial _{t}+\zeta \in \mathfrak{X}\left( \bar{M}\right) $%
. Assume that $H^{f}=0$ and $\left( M,g\right) $ is Einstein with factor $%
-\left( n-1\right) \dot{f}^{2}$. Then $\bar{\zeta}$ is conformal with factor 
$2\lambda $.
\end{theorem}

\begin{proof}
Let $\left( \bar{M},\bar{g},\bar{\zeta},\lambda \right) $ be a Ricci soliton
where $\bar{M}=I\times _{f}M$ is a generalized Robertson-Walker spacetime
and $\bar{\zeta}=h\partial _{t}+\zeta \in \mathfrak{X}\left( \bar{M}\right) $%
. Then%
\begin{equation*}
\frac{1}{2}\left( \mathcal{L}_{\bar{\zeta}}\bar{g}\right) \left( \bar{X},%
\bar{Y}\right) +\mathrm{\bar{R}ic}\left( \bar{X},\bar{Y}\right) =\lambda 
\bar{g}\left( \bar{X},\bar{Y}\right) 
\end{equation*}%
for any vector fields $\bar{X},\bar{Y}\in \mathfrak{X}\left( \bar{M}\right) $%
. Equation (\ref{Ricci}) implies that%
\begin{eqnarray*}
\mathrm{\bar{R}ic}\left( \bar{X},\bar{Y}\right)  &=&\mathrm{\bar{R}ic}\left(
x\partial _{t},y\partial _{t}\right) +\mathrm{\bar{R}ic}\left( X,Y\right)  \\
&=&-\frac{n}{f}H^{f}\left( x\partial _{t},y\partial _{t}\right) +\mathrm{Ric}%
\left( X,Y\right) -f^{\diamond }g\left( X,Y\right)  \\
&=&\mathrm{Ric}\left( X,Y\right) -f^{\diamond }g\left( X,Y\right)  \\
&=&\mathrm{Ric}\left( X,Y\right) +\left( n-1\right) \dot{f}^{2}g\left(
X,Y\right) =0
\end{eqnarray*}%
Thus%
\begin{equation*}
\left( \mathcal{L}_{\bar{\zeta}}\bar{g}\right) \left( \bar{X},\bar{Y}\right)
=2\lambda \bar{g}\left( \bar{X},\bar{Y}\right) 
\end{equation*}%
i.e. $\bar{\zeta}$ is a conformal vector field with factor $2\lambda $.
\end{proof}

\begin{theorem}
Let $\left( \bar{M},\bar{g},\bar{\zeta},\lambda \right) $ be a Ricci soliton
where $\bar{M}=I\times _{f}M$ is a generalized Robertson-Walker spacetime
and $\bar{\zeta}=h\partial _{t}+\zeta \in \mathfrak{X}\left( \bar{M}\right) $
be a concircular vector field on $\overline{M}$ with factor one. Then $%
\left( M,g\right) $ is Ricci flat if $-\left[ \left( n+1\right) f\ddot{f}%
+\left( n-1\right) \dot{f}^{2}\right] $ is constant.
\end{theorem}

\begin{proof}
A concircular vector field is conformal, and so the above theorem implies
that%
\begin{equation}
{\mathrm{Ric}}\left( X,Y\right) =-\left[ \left( n+1\right) f\ddot{f}+\left(
n-1\right) \dot{f}^{2}\right] g\left( X,Y\right) 
\end{equation}%
for any vector fields $\bar{X},\bar{Y}\in \mathfrak{X}\left( \bar{M}\right) $%
. Since $\bar{\zeta}$ is concircular, Theorems (\ref{concircular1}) and (\ref%
{concircular2}) yield%
\begin{equation*}
\mathrm{Ric}\left( \zeta ,\zeta \right) =0
\end{equation*}%
Therefore, for a constant factor $\mu =-\left[ \left( n+1\right) f\ddot{f}%
+\left( n-1\right) \dot{f}^{2}\right] $ we have%
\begin{eqnarray*}
{\mathrm{Ric}}\left( X,Y\right)  &=&\mu g\left( X,Y\right)  \\
\mathrm{Ric}\left( \zeta ,\zeta \right)  &=&0 \\
\mu  &=&0
\end{eqnarray*}%
i.e. $\left( M,g\right) $ is Ricci flat.
\end{proof}

\begin{theorem}
Let $\bar{\zeta}=h\partial _{t}+\zeta \in \mathfrak{X}\left( \bar{M}\right) $
be a vector field on a generalized Robertson-Walker spacetime $\overline{M}%
=I\times _{f}M$ Then $\left( \bar{M},\bar{g},\bar{\zeta},\lambda \right) $
is a Ricci soliton if
\end{theorem}

\begin{enumerate}
\item $\zeta $ is a conformal vector field with conformal factor $2\rho $,

\item $h\partial _{t}$ is a conformal vector field with conformal factor $%
2\sigma $,

\item $\left( M,g\right) $ is Einstein with factor $\mu $, and

\item $\left( \sigma -\rho \right) f^{2}=\mu +h\dot{f}f+\left( n+1\right) f%
\ddot{f}+\left( n-1\right) \dot{f}^{2}$
\end{enumerate}

Moreover, $\lambda =\sigma -\frac{n\ddot{f}}{f}$.

\begin{proof}
Let $\bar{\zeta}=h\partial _{t}+\zeta \in \mathfrak{X}\left( \bar{M}\right) $%
, then%
\begin{equation*}
\frac{1}{2}\left( \mathcal{L}_{\bar{\zeta}}\bar{g}\right) \left( \bar{X},%
\bar{Y}\right) =\frac{1}{2}\left( \left( \mathcal{L}_{h\partial
_{t}}g\right) (x\partial _{t},y\partial _{t})+f^{2}\left( \mathcal{L}_{\zeta
}g\right) (X,Y)+2hf\dot{f}g(X,Y)\right)
\end{equation*}%
and%
\begin{eqnarray*}
\mathrm{\bar{R}ic}\left( \bar{X},\bar{Y}\right) &=&{\mathrm{Ric}}\left(
X,Y\right) -\frac{n}{f}H^{f}(x\partial _{t},y\partial _{t})-f^{\diamond
}g\left( X,Y\right) \\
&=&{\mathrm{Ric}}\left( X,Y\right) -\frac{n\ddot{f}}{f}g_{I}(x\partial
_{t},y\partial _{t})-f^{\diamond }g\left( X,Y\right)
\end{eqnarray*}%
Since $\zeta $ is a conformal vector field with conformal factor $2\rho $
and $h\partial _{t}$ is a conformal vector field with conformal factor $%
2\sigma $,%
\begin{eqnarray*}
\frac{1}{2}\left( \mathcal{L}_{\bar{\zeta}}\bar{g}\right) \left( \bar{X},%
\bar{Y}\right) &=&\frac{1}{2}\left( \left( \mathcal{L}_{h\partial
_{t}}g\right) (x\partial _{t},y\partial _{t})+f^{2}\left( \mathcal{L}_{\zeta
}g\right) (X,Y)+2hf\dot{f}g(X,Y)\right) \\
&=&\sigma g_{I}(x\partial _{t},y\partial _{t})+\rho f^{2}g(X,Y)+hf\dot{f}%
g(X,Y) \\
&=&\sigma g_{I}(x\partial _{t},y\partial _{t})+\left( \rho f^{2}+hf\dot{f}%
\right) g(X,Y)
\end{eqnarray*}%
and%
\begin{eqnarray*}
\mathrm{\bar{R}ic}\left( \bar{X},\bar{Y}\right) &=&\mu g\left( X,Y\right) -%
\frac{n\ddot{f}}{f}g_{I}(x\partial _{t},y\partial _{t})-f^{\diamond }g\left(
X,Y\right) \\
&=&-\frac{n\ddot{f}}{f}g_{I}(x\partial _{t},y\partial _{t})+\left( \mu
-f^{\diamond }\right) g\left( X,Y\right)
\end{eqnarray*}%
where $f^{\diamond }=-f\ddot{f}-\left( n-1\right) \dot{f}^{2}$. Thus

\begin{eqnarray*}
&&\frac{1}{2}\left( \mathcal{L}_{\bar{\zeta}}\bar{g}\right) \left( \bar{X},%
\bar{Y}\right) +\mathrm{\bar{R}ic}\left( \bar{X},\bar{Y}\right) \\
&=&\left( \sigma -\frac{n\ddot{f}}{f}\right) g_{I}(x\partial _{t},y\partial
_{t})+\left( \rho f^{2}+hf\dot{f}+\mu -f^{\diamond }\right) g(X,Y) \\
&=&\left( \sigma -\frac{n\ddot{f}}{f}\right) g_{I}(x\partial _{t},y\partial
_{t})+f^{2}\left( \rho +\frac{h\dot{f}}{f}+\frac{\mu -f^{\diamond }}{f^{2}}%
\right) g(X,Y)
\end{eqnarray*}%
The last condition implies that%
\begin{equation*}
\sigma -\frac{n\ddot{f}}{f}=\rho +\frac{h\dot{f}}{f}+\frac{\mu -f^{\diamond }%
}{f^{2}}=\lambda
\end{equation*}%
and so%
\begin{eqnarray*}
\frac{1}{2}\left( \mathcal{L}_{\bar{\zeta}}\bar{g}\right) \left( \bar{X},%
\bar{Y}\right) +\mathrm{\bar{R}ic}\left( \bar{X},\bar{Y}\right) &=&\lambda %
\left[ g_{I}(x\partial _{t},y\partial _{t})+f^{2}g(X,Y)\right] \\
&=&\lambda \bar{g}\left( \bar{X},\bar{Y}\right)
\end{eqnarray*}%
and the proof is complete.
\end{proof}


\begin{thebibliography}{99}
\bibitem{Barros:2012} A. Barros and E. Ribeiro Jr., \emph{Some
characterizations for compact almost Ricci solitons}, Proc. Amer. Math. Soc. 
\textbf{140}(2012), 1033-1040.

\bibitem{Barros:2013} A. Barros, Jos\'{e} N. Gomes, E. Ribeiro Jr. \emph{A
note on rigidity of the almost Ricci soliton, }Archiv der Mathematik, 
\textbf{100(}2013), no. 5, 481-490

\bibitem{Berestovskii2008} V. N. Berestovskii, Yu. G. Nikonorov, \emph{%
Killing vector fields of constant length on Riemannian manifolds}, Siberian
Mathematical Journal, \textbf{49}(2008), Issue 3 , pp 395-407.

\bibitem{Bernstein:2015} Jacob Bernstein and Thomas Mettler, \emph{%
Two-Dimensional Gradient Ricci Solitons Revisited}, Int Math Res Not, 
\textbf{2015}(2015), no. 1, Pp. 78-98

\bibitem{Besse2008} A. L. Besse\textbf{, }\emph{Einstein Manifolds},
Classics in Mathematics, Springer-Verlag, Berlin, 2008.

\bibitem{Bishop1969} R. L. Bishop and B. O'Neill, \emph{Manifolds of
negative curvature}, Trans. Amer. Math. Soc. \textbf{145} (1969), 1-49.

\bibitem{Caballero2011} M. Caballero, A. Romero and R. M. Rubio, \emph{%
Constant Mean Curvature Spacelike Hypersurfaces in Lorentzian Manifolds with
a Timelike Gradient Conformal Vector Field}, Classical and Quantum Gravity, 
\textbf{28}(2011), no. 14, 145009(1-14).

\bibitem{Chen2008} B.-Y. Chen and S. W. Wei, \emph{Differential Geometry of
Submanifolds of Warped Product Manifolds} $I\times _{f}S^{m-1}\left(
k\right) $, Journal Geometry, \textbf{91}(2008), 21--42.

\bibitem{Chen2014} B.-Y. Chen, \emph{A Simple Characterization Of
Generalized Robertson-Walker Spacetimes}, General Relativity and
Gravitation, \textbf{46}(2014), 1833-9.

\bibitem{Deshmokh2012} S. Deshmukh , \emph{Conformal vector fields and eigen
vectors of the Laplacian operator, }Math. Phys. Anal. Geo., \textbf{15}%
(2012), pp. 163-172.

\bibitem{Deshmokh20141} S. Deshmukh and F. R. Al-Solamy, \emph{Conformal
vector fields on a Riemannian manifold, }Balkan Journal of Geometry and Its
Applications, \textbf{19}(2014), No.2, pp. 86-93.

\bibitem{Deshmokh20142} S. Deshmukh and F. R. Al-Solamy, \emph{A note on
conformal vector fields on a Riemannian manifold, }Colloq. Math. \textbf{136}
(2014), 65-73.

\bibitem{Unal2012} F. Dobarro, B. Unal, \emph{Characterizing killing vector
fields of standard static space-times}, J. Geom. Phys. \textbf{62} (2012),
1070--1087.

\bibitem{Duggala2005} K. L. Duggala and R. Sharma,\emph{Conformal killing
vector fields on spacetime solutions of Einstein's equations and initial data%
}, Nonlinear Analysis \textbf{63} (2005), 447-454.

\bibitem{Fernandez:2011} M. Fern\'{a}ndez-L\'{o}pez, Eduardo Garc\'{\i}a-R%
\'{\i}o, \emph{Rigidity of shrinking Ricci solitons}, Mathematische
Zeitschrift, 2011, Volume 269, Issue 1-2, pp 461-466

\bibitem{Sanchez2000} J. L. Flores and M. S\'{a}nchez, \emph{\ Geodesic
Connectedness and Conjugate Points in GRW Space-times}, J. Geom. Phys., 
\textbf{36} (2000), no.3-4, 285-314.

\bibitem{Garcia:2007} Alberto A. Garcia, Steve Carlip, $n-$\emph{dimensional
generalizations of the Friedmann--Robertson--Walker cosmology}, Physics
Letters B 645 (2007) 101--107.

\bibitem{Hall2004} G. S. Hall, \emph{Symmetries and Curvature Structure in
General Relativity}, World Scientific Publishing Co. Ltd, London, 2004.

\bibitem{Ivancevic2007} V.\textbf{\ }G. Ivancevic and T. T. Ivancevic, \emph{%
Applied Differential Geometry: A Modern Introduction}, World Scientific
Publishing Co. Ltd, London, 2007.

\bibitem{Peterson:2009} P. Petersen and W. Wylie, \emph{Rigidity of gradient
Ricci solitons}, Pacific Journal of Mathematics, \textbf{241}(2009), no. 2,
329-345.

\bibitem{Kuhnel1997} W. Kuhnel and H. Rademacher, \emph{Conformal vector
fields on pseudo-Riemannian spaces}, Journal of Geometry and its
Applications, \textbf{7}(1997), 237--250.

\bibitem{Munteanu:2013} Ovidiu Munteanu, Natasa Sesum, \emph{On Gradient
Ricci Solitons}, Journal of Geometric Analysis, \textbf{23(}2013), no. 2, pp
539-561

\bibitem{Nomizo1960} K. Nomizu, \emph{On local and global existence of
Killing vector fields}, Ann. of Math. \textbf{72}(1960), no. 1, 105--120.

\bibitem{Oneill1983} B. O'Neill, \emph{Semi-Riemannian Geometry with
Applications to Relativity}, Academic Press Limited, London, 1983.

\bibitem{Sanchez1998} M. S\'{a}nchez, \emph{\ On the Geometry of Generalized
Robertson-Walker Spacetimes: geodesics}, Gen. Relativ. Gravitation, \textbf{%
30} (1998), no.6, 915-932.

\bibitem{Sanchez1999} M. S\'{a}nchez, \emph{On the Geometry of Generalized
Robertson-Walker Spacetimes: Curvature and Killing fields}, J. Geom. Phys., 
\textbf{31} (1999), No.1, 1-15.

\bibitem{Shenawy:2016} S. Shenawy and B. Unal. \emph{The }$W_{2}-$\emph{%
curvature tensor on warped product manifolds and applications},
International Journal of Geometric methods in Mathematical Physics, to
appear.

\bibitem{Shenawy:2015} S. Shenawy and B. Unal. $2-$\emph{Killing vector
fields on warped product manifolds}, International Journal of Mathematics, 
\textbf{26}(2015), no. 8, 1550065(17 pages).

\bibitem{Steller2006} M. Steller, \emph{Conformal vector fields on
spacetimes, }Ann Glob Anal Geom \textbf{29}(2006), 293--317.

\bibitem{Yorozu1982} S. Yorozu, \emph{Killing Vector Fields on Non-Compact
Riemannian Manifolds with Boundary}, Kodai Math. J. \textbf{5}(1982),
426--433.
\end{thebibliography}
\end{document}